\documentclass[11pt, reqno]{amsart}
\usepackage{amsmath, amsthm, amscd, amsfonts, amssymb, graphicx, color}

\begin{document}

\title[Permuting triderivations and permuting trihomomorphisms]{Permuting triderivations and permuting trihomomorphisms in complex Banach algebras}

\author[C. Park]
{Choonkil Park}

\address{\noindent Choonkil Park \newline \indent  Research Institute for Natural Sciences, Hanyang University, Seoul 04763, Korea}
\email{baak@hanyang.ac.kr}

\begin{abstract}
In this paper, we solve the following  tri-additive $s$-functional inequalities
\begin{eqnarray}\label{0.1} && \nonumber \| f(x+y, z-w, a+b) + f(x-y, z+w, a-b) \\ &&  \nonumber\qquad -2 f(x, z, a) + 2 f(x, w, b) -2f(y, z, b) +2 f(y, w, a)\| \\ &&  \quad  \le \left \|s \left(2f\left(\frac{x+y}{2}, z-w, a+b \right) + 2f\left(\frac{x-y}{2}, z+w, a-b\right) \right. \right. \\ && \qquad \left. \left. -2 f(x, z, a) + 2 f(x, w, b) -2f(y, z, b) +2 f(y, w, a)\right)\right\| , \nonumber
\end{eqnarray}
\begin{eqnarray}\label{0.2} && \nonumber \left\|2f\left(\frac{x+y}{2}, z-w, a+b \right) + 2f\left(\frac{x-y}{2}, z+w, a-b\right) \right.  \\ && \nonumber \qquad \left. -2 f(x, z, a) + 2 f(x, w, b) -2f(y, z, b) +2 f(y, w, a)\right\| \\ && \quad  \le \|s ( f(x+y, z-w, a+b) + f(x-y, z+w, a-b) \\ &&  \nonumber\qquad -2 f(x, z, a) + 2 f(x, w, b) -2f(y, z, b) +2 f(y, w, a) )\| , 
\end{eqnarray}
where $s$ is a fixed nonzero complex number with $|s |<  1$.
Moreover, we prove the Hyers-Ulam stability and hyperstability of   permuting triderivations and permuting  trihomomorphisms in  Banach algebras and unital $C^*$-algebras, 
associated with  the tri-additive $s$-functional inequalities  {\rm (\ref{0.1})} and {\rm (\ref{0.2})}.
\end{abstract}


\keywords{permuting triderivation on $C^*$-algebra; permuting  trihomomorphism in Banach algebra; 
 Hyers-Ulam stability; hyperstability;  tri-additive $s$-functional inequality.\\ $^*$Corresponding author: Choonkil Park (email: baak@hanyang.ac.kr; fax: +82-2-2281-0019; office: +82-2-2220-0892.}

\baselineskip=13.4pt

\maketitle

\newtheorem{theorem}{Theorem}[section]
\newtheorem{lemma}[theorem]{Lemma}
\newtheorem{proposition}[theorem]{Proposition}
\newtheorem{corollary}[theorem]{Corollary}
\theoremstyle{definition}
\newtheorem{definition}[theorem]{Definition}
\newtheorem{example}[theorem]{Example}
\newtheorem{xca}[theorem]{Exercise}
\newtheorem{problem}[theorem]{Problem}
\theoremstyle{remark}
\newtheorem{remark}[theorem]{Remark}
\numberwithin{equation}{section}

\section{Introduction and preliminaries}\label{sec:1}

The stability problem of functional equations  originated from a
question of Ulam \cite{ul60}  concerning the stability of group
homomorphisms.
 Hyers \cite{hy41} gave a first affirmative partial answer to the question of Ulam for Banach spaces.
Hyers' Theorem  was generalized by Aoki \cite{a} for additive
mappings and by  Rassias \cite{ra78} for linear mappings by
considering an unbounded Cauchy difference.  A
generalization of the  Rassias theorem was obtained by G\u
avruta \cite{ga94} by replacing the unbounded Cauchy difference by a
general control function in the spirit of  Rassias' approach.
  Park \cite{ppp, ppp1} defined additive $\rho$-functional inequalities and proved the Hyers-Ulam stability of the additive $\rho$-functional inequalities in Banach spaces and non-Archimedean Banach spaces. The
 stability problems of various functional equations and functional inequalities have been extensively investigated
 by a number of authors (see \cite{c16, e17, efp, ega,  eg, kiss, lm}).

 \"{O}zt\"{u}rk \cite{o99} introduced and investigated  permuting triderivations in rings.
\"{O}zden {\it et al.} \cite{o06},  Yazarli \cite{y17} and Yazarli {\it et al. } \cite{y05}  investigated some properties of permuting triderivations on rings.

\begin{definition} {\rm \cite{o99}} Let $A$ be a ring. A tri-additive mapping $D: A^3 \to A$
is called a {\it permuting triderivation} on $A$ if $D$ satisfies 
\begin{eqnarray*}
D (xy,z, w) & = &   D(x,z, w)y + xD(y,z, w) ,\\
D(x_{\sigma(1)},x_{\sigma(2)}, x_{\sigma(3)}) &= & D(x_1, x_2, x_3)
\end{eqnarray*}
for all $x, y, z, w, x_1, x_2, x_3\in A$ and for every permutation  $(\sigma(1), \sigma(2), \sigma(3))$ of $(1,2,3)$.
\end{definition}

In this paper, we introduce  biderivations and bihomomorphisms  in a Banach algebra.

\begin{definition} Let $A$ and $B$ be complex Banach algebras. A $\mathbb{C}$-trilinear mapping $D: A^3 \to A$
is called a {\it  permuting triderivation} on $A$ if $D$ satisfies 
\begin{eqnarray*}
D (xy,z, w) & = &   D(x,z, w)y + xD(y,z, w) ,\\
D(x_{\sigma(1)},x_{\sigma(2)}, x_{\sigma(3)}) &= & D(x_1, x_2, x_3)
\end{eqnarray*}
for all $x, y, z, w, x_1, x_2, x_3\in A$ and for every permutation  $(\sigma(1), \sigma(2), \sigma(3))$ of $(1,2,3)$.

 A $\mathbb{C}$-trilinear mapping $H: A^3 \to B$
is called a {\it permuting   trihomomorphism}  if $H$ satisfies 
\begin{eqnarray*}
H (xy,zw, ab) & = &   H(x,z,a)H(y,w,b) ,\\
H(x_{\sigma(1)},x_{\sigma(2)}, x_{\sigma(3)}) &= & H(x_1, x_2, x_3)
\end{eqnarray*}
for all $x, y, z, w, a, b, x_1, x_2, x_3\in A$ and for every permutation  $(\sigma(1), \sigma(2), \sigma(3))$ of $(1,2,3)$.
\end{definition}

This paper is organized as follows:  In Sections 2 and 3, we solve the tri-additive $s$-functional inequalities {\rm (\ref{0.1})} and  {\rm (\ref{0.2})}  and prove the  Hyers-Ulam stability of the tri-additive $s$-functional inequalities {\rm (\ref{0.1})} and  {\rm (\ref{0.2})}  in complex Banach spaces.
In Sections 4 and 5, we investigate permuting  triderivations and permuting  trihomomorphisms  in Banach algebras and unital $C^*$-algebras associated with the tri-additive $s$-functional inequalities  {\rm (\ref{0.1})} and  {\rm (\ref{0.2})}.

Throughout this paper, let $X$ be a complex normed space and $Y$ be a complex Banach space. Let  $A$ and $B$ be  complex Banach algebras. 
 Assume that $s$ is a fixed nonzero complex number with $|s| < 1$.

\section{Tri-additive $s$-functional inequality {\rm (\ref{0.1})}}\label{sec:2}

We solve and  investigate the tri-additive $s$-functional inequality  (\ref{0.1})  in  complex  normed spaces.

\begin{lemma}\label{lm2.2}
If a mapping  $f : X^3 \rightarrow Y$ satisfies $f(0, z, a) = f(x,0, a) = f(x,z, 0)=0$ and  
\begin{eqnarray}\label{2.1}&& \nonumber \| f(x+y, z-w, a+b) + f(x-y, z+w, a-b) \\ &&  \nonumber\qquad -2 f(x, z, a) + 2 f(x, w, b) -2f(y, z, b) +2 f(y, w, a)\| \\ &&  \quad  \le \left \|s \left(2f\left(\frac{x+y}{2}, z-w, a+b \right) + 2f\left(\frac{x-y}{2}, z+w, a-b\right) \right. \right. \\ && \qquad \left. \left. -2 f(x, z, a) + 2 f(x, w, b) -2f(y, z, b) +2 f(y, w, a)\right)\right\| , \nonumber
\end{eqnarray} for all $x, y, z, w, a, b \in X$, then  $f : X^3 \rightarrow Y$ is tri-additive.
\end{lemma}

\begin{proof}
Assume that $f: X^3 \to Y$ satisfies (\ref{2.1}). 

Letting $x=y$ and $w=b=0$ in {\rm (\ref{2.1})}, we get
$
f(2x, z, a) = 2f(x, z, a)
$
 for all $x, z, a\in X$.

Letting $w=b=0$ in {\rm (\ref{2.1})}, we get
$
f(x+y, z, a) + f(x-y, z, a) = 2 f(x, z, a)
$
and so $f(x_1, z, a) + f(y_1, z, a) = 2 f\left(\frac{x_1 + y_1}{2}, z, a\right) = f(x_1 + y_1, z, a)$ for all $x_1 : = x+y, y_1 : = x-y, z, a \in X$, since $|s|< 1$ and $f(0, z, a) =0$ for all $z, a\in X$. So $f: X^3 \to Y$ is additive in the first variable. 

Similarly, one can show that $f: X^3 \to Y$ is additive in the second variable and in the third variable. Hence $f: X^3 \to Y$ is tri-additive. 
\end{proof}

We prove the Hyers-Ulam stability of the tri-additive $s$-functional
inequality (\ref{2.1}) in  complex Banach spaces.

\begin{theorem}\label{thm2.3}
Let $r > 1$ and $\theta$ be nonnegative real numbers   and let $f :
X^3 \rightarrow Y$ be a  mapping satisfying $f(x, 0, a)= f(0, z, a)= f(x,z, 0)=0$ and 
\begin{eqnarray}\label{2.4}&& \nonumber \| f(x+y, z-w, a+b) + f(x-y, z+w, a-b) \\ &&  \nonumber\qquad -2 f(x, z, a) + 2 f(x, w, b) -2f(y, z, b) +2 f(y, w, a)\| \\ &&  \quad  \le \left \|s \left(2f\left(\frac{x+y}{2}, z-w, a+b \right) + 2f\left(\frac{x-y}{2}, z+w, a-b\right) \right. \right. \\ && \qquad \left. \left. -2 f(x, z, a) + 2 f(x, w, b) -2f(y, z, b) +2 f(y, w, a)\right)\right\| , \nonumber 
  \\ & &  \qquad +  
 \theta ( \|x\|^r + \|y\|^r) (\|z\|^r +\|w\|^r  ) (\|a\|^r +\|b\|^r  ) \nonumber \end{eqnarray}
for all $x, y, z, w, a, b \in X$.
Then there exists a unique  tri-additive mapping $L : X^3 \rightarrow Y$
such that
\begin{eqnarray}\label{2.5}
\|f(x, z, a)- L(x, z, a) \|  \le  \frac{2 \theta }{2^{r}-2}  \|x \|^r \|z\|^r \|a\|^r 
\end{eqnarray} for all $x, z, a \in X$.
\end{theorem}

\begin{proof}
Letting $w=b=0$ and  $y = x$  in {\rm (\ref{2.4})}, we get
\begin{eqnarray}\label{2.6}
 \| f(2x, z, a) - 2f(x, z, a) \|  \le 2 \theta \|x\|^{r} \|z\|^r \|a\|^r 
\end{eqnarray}
for all $x, z, a \in X$.  So
\begin{eqnarray*}
  \left\| f(x, z, a) - 2 f\left(\frac{x}{2}, z, a\right) \right\|  \le  \frac{2}{2^{r}} \theta \|x\|^{r}\|z\|^r \|a\|^r 
\end{eqnarray*}
for all $x, z, a \in X$.
 Hence
\begin{eqnarray}\label{2.7}
\left\|2^{l} f\left(\frac{x}{2^l}, z, a\right) - 2^{m} f\left(\frac{x}{2^{m}}, z, a\right)\right\|   & \le & 
 \sum_{j=l}^{m-1}\left\| 2^j f\left(\frac{x}{2^j}, z, a\right) - 2^{j+1} f\left(\frac{x}{2^{j+1}}, z, a\right) \right\| 
 \\ & \le  & \nonumber \frac{2 }{2^{ r}} \sum_{j=l}^{m-1} \frac{2^{ j}}{2^{ rj}} \theta \|x\|^{r} \|z\|^r \|a\|^r 
\end{eqnarray}
 for all nonnegative
integers $m$ and $l$ with $m>l$ and all $x, z, a \in X$. It follows from
{\rm (\ref{2.7})} that the sequence
$\{2^{k} f(\frac{x}{2^k}, z, a )\}$ is  Cauchy  for all $x, z, a \in X$.
Since $Y$ is a  Banach space, the sequence $\{2^{k} f(\frac{x}{2^k}, z, a)\}$ converges.
So one can define
the mapping $L : X^3 \rightarrow Y$ by
$$L(x, z,a) : = \lim_{k\to \infty} 2^{k} f\left(\frac{x}{2^k}, z, a\right) $$
for all $x, z, a \in X$. Moreover, letting $l =0$ and passing the limit $m \to
\infty$ in {\rm (\ref{2.7})}, we get {\rm (\ref{2.5})}.

It follows from {\rm (\ref{2.4})} that
\begin{eqnarray*}
&&  \| L(x+y, z-w, a+b) + L(x-y, z+w, a-b) \\ &&  \qquad -2 L(x, z, a) + 2 L(x, w, b) -2L(y, z, b) +2 L(y, w, a)\|  \\ && \quad = \lim_{n\to\infty}  \left\| 2^n \left(
f\left(\frac{x+y}{2^{n}}, z-w, a+b\right) +  f\left(\frac{x-y}{2^{n}}, z+w, a-b\right) -2 f\left(\frac{x}{2^{n}}, z, a\right)  \right. \right. \\ && \qquad \qquad \left.\left. +2 f\left(\frac{x}{2^{n}}, w, b\right)  -2 f\left(\frac{y}{2^{n}}, z, b \right) + 2 f\left(\frac{y}{2^{n}}, w, a \right)\right) \right\| \\ && \quad   \le   \lim_{n\to\infty}   \left\| 2^n s \left(
 2f\left(\frac{x+y}{2^{n}}, z-w, a+b \right) +  2 f\left(\frac{x-y}{2^{n}}, z+w, a-b\right) -2 f\left(\frac{x}{2^{n}}, z, a\right)  \right. \right. \\ && \qquad \qquad \left.\left. +2 f\left(\frac{x}{2^{n}}, w, b\right)  -2 f\left(\frac{y}{2^{n}}, z, b \right) + 2 f\left(\frac{y}{2^{n}}, w, a \right)\right) \right\| \\  \\ && \qquad\qquad   +    \lim_{n\to \infty} 
\frac{2^{n}}{2^{ r n}}  \theta (\|x\|^r + \|y\|^r ) ( \|z\|^r + \|w\|^r )  (\|a\|^r +\|b\|^r  ) 
\\ && \quad  \le \left \|s \left(2L\left(\frac{x+y}{2}, z-w, a+b \right) + 2L\left(\frac{x-y}{2}, z+w, a-b\right) \right. \right. \\ && \qquad \left. \left. -2 L(x, z, a) + 2 L(x, w, b) -2L(y, z, b) +2 L(y, w, a)\right)\right\|  
\end{eqnarray*}
for all $x, y, z, w, a, b \in X$. So
\begin{eqnarray*} &&  \| L(x+y, z-w, a+b) + L(x-y, z+w, a-b) \\ &&  \qquad -2 L(x, z, a) + 2 L(x, w, b) -2L(y, z, b) +2 L(y, w, a)\| \\ &&  \quad  \le \left \|s \left(2L\left(\frac{x+y}{2}, z-w, a+b \right) + 2L\left(\frac{x-y}{2}, z+w, a-b\right) \right. \right. \\ && \qquad \left. \left. -2 L(x, z, a) + 2 L(x, w, b) -2L(y, z, b) +2 L(y, w, a)\right)\right\|    \end{eqnarray*}
 for all $x, y, z, w, a, b
\in X$. By Lemma \ref{lm2.2}, the mapping $L : X^3 \rightarrow Y$ is tri-additive.

Now, let $T : X^3 \rightarrow Y$ be another
 tri-additive mapping satisfying {\rm (\ref{2.5})}. Then we have
\begin{eqnarray*}
 && \|L(x, z, a) - T(x, z, a)\|    =     \left\| 2^{q}L\left(\frac{x}{2^{q}}, z, a\right) - 2^{q} T\left(\frac{x}{2^{q}}, z, a\right)\right\| \\
&& \qquad  \le    \left\| 2^qL\left(\frac{x}{2^{q}}, z, a\right) -  2^qf\left(\frac{x}{2^{q}}, z, a\right) \right\| + 
 \left\| 2^q T\left( \frac{x}{2^{q}}, z, a\right)- 2^q f\left(\frac{x}{2^{q}}, z, a\right) \right\|  \\ &&\qquad 
  \le   \frac{4  \theta }{2^{ r}-2} \frac{2^{ q}}{ 2^{ qr}} \|x \|^r \|z \|^r \|a\|^r,
\end{eqnarray*}
which tends to zero as $q \to \infty$ for all $x, z, a \in X$. So we can
conclude that $L(x, z, a)=T(x, z, a)$ for all $x, z, a \in X$. This proves the
uniqueness of $L$, as desired.
\end{proof}

\begin{theorem}
Let $r < 1$ and $\theta$ be nonnegative real numbers  and let $f :
X^3 \rightarrow Y$ be a mapping satisfying   {\rm (\ref{2.4})} and $f(x,0, a)=f(0,z, a)=f(x,z, 0)=0$ for all $x,z,a\in X$. 
Then there exists a unique  tri-additive mapping $L : X^3 \rightarrow Y$
such that
\begin{eqnarray}\label{2.8}
\|f(x, z, a) - L(x, z, a) \|  \le  \frac{ 2 \theta }{2- 2^{r}} \|x \|^r \|z\|^r \|a\|^r
\end{eqnarray} for all $x, z, a \in X$.
\end{theorem}

\begin{proof}
It follows from {\rm (\ref{2.6})} that
\begin{eqnarray*}
  \left\| f(x, z, a) - \frac{1}{2} f(2x, z, a) \right\|  \le \theta \|x\|^r  \|z\|^r  \|a\|^r
\end{eqnarray*}
for all $x, z, a  \in X$.
Hence
 \begin{eqnarray}\label{2.9}
\left\|\frac{1}{2^l} f(2^l x, z,a) - \frac{1}{2^m} f(2^m x, z, a)\right\| &\le &
 \sum_{j=l}^{m-1} \left\| \frac{1}{2^j} f\left(2^j x, z, a\right) -\frac{1}{ 2^{j+1}} f\left(2^{j+1} x, z, a\right) \right\| \\ & \le & \nonumber    \sum_{j=l}^{m-1} \frac{2^{ r j}}{2^{ j+1}} \theta \|x\|^{r}  \|z\|^r \|a\|^r
\end{eqnarray}
 for all nonnegative
integers $m$ and $l$ with $m>l$ and all $x, z,a \in X$.
 It follows from  {\rm (\ref{2.9})} that the sequence $\{\frac{1}{2^n}
f(2^n x, z,a)\}$ is a Cauchy sequence for all $x, z,a \in X$. Since
$Y$ is complete, the sequence $\{\frac{1}{2^n} f(2^n x, z,a)\}$
converges. So one can define the mapping $L : X^3 \rightarrow Y$ by
$$L(x, z,a) : = \lim_{n\to \infty} \frac{1}{2^n} f(2^n x, z,a) $$
for all $x, z,a \in X$. Moreover, letting $l =0$ and passing the limit $m \to
\infty$ in {\rm (\ref{2.9})}, we get {\rm (\ref{2.8})}.

The rest of the proof is similar to the proof of Theorem \ref{thm2.3}.
\end{proof}

\section{Tri-additive $s$-functional inequality {\rm (\ref{0.2})}}\label{sec:3}

We solve and  investigate the tri-additive $s$-functional inequality  (\ref{0.2})  in  complex  normed spaces.

\begin{lemma}\label{lm3.1}
If a mapping  $f : X^3 \rightarrow Y$ satisfies $f(0, z, a) = f(x,0, a)=f(x,z,0)=0$ and  
\begin{eqnarray}\label{3.1} && \nonumber \left\|2f\left(\frac{x+y}{2}, z-w, a+b \right) + 2f\left(\frac{x-y}{2}, z+w, a-b\right) \right.  \\ && \nonumber \qquad   -2 f(x, z, a) + 2 f(x, w, b) -2f(y, z, b) +2 f(y, w, a) \| \\ &&  \quad  \le \left \|s \left(      f(x+y, z-w, a+b) + f(x-y, z+w, a-b) \right.\right.\\ &&  \nonumber\qquad \left.\left. -2 f(x, z, a) + 2 f(x, w, b) -2f(y, z, b) +2 f(y, w, a) \right)\right\|  \nonumber
\end{eqnarray} for all $x, y, z, w, a, b \in X$, then  $f : X^3 \rightarrow Y$ is tri-additive.
\end{lemma}

\begin{proof}
Assume that $f: X^3 \to Y$ satisfies (\ref{3.1}). 

Letting $y=w=b=0$ in {\rm (\ref{3.1})}, we get
$
4 f\left(\frac{x}{2}, z, a\right) = 2f(x, z, a)
$
 for all $x, z, a\in X$. 

Let $w=b=0$ in {\rm (\ref{3.1})}.  It follows from {\rm (\ref{3.1})} that 
\begin{eqnarray*}
&&  \|f(x+y, z, a ) + f\left(x-y, z, a\right)   -2 f(x, z, a)  \| 
\\ && = \left\|2f\left(\frac{x+y}{2}, z, a \right) + 2f\left(\frac{x-y}{2}, z, a\right) -2 f(x, z, a) \right\| \\ &&  \le \left\|s (      f(x+y, z, a) + f(x-y, z, a)  -2 f(x, z, a) )\right\|
\end{eqnarray*}
and so $f(x+y, z, a ) + f\left(x-y, z, a\right)   -2 f(x, z, a)$ for all $x,y, z, a \in X$, since $|s|< 1$. So $f: X^3 \to Y$ is additive in the first variable. 

Similarly, one can show that $f: X^3 \to Y$ is additive in the second variable and in the third variable. Hence $f: X^3 \to Y$ is tri-additive. 
\end{proof}

We  prove the Hyers-Ulam stability of the tri-additive $s$-functional
 inequality (\ref{3.1}) in   complex Banach spaces.

\begin{theorem}
Let $r > 1$ and $\theta$ be nonnegative real numbers    and let $f :
X^3 \rightarrow Y$ be a  mapping satisfying $f(x,0,a)= f(0,z,a) =f(x,z,0)=0$ and 
\begin{eqnarray} \label{3.4}
 && \nonumber \left\|2f\left(\frac{x+y}{2}, z-w, a+b \right) + 2f\left(\frac{x-y}{2}, z+w, a-b\right) \right.  \\ && \nonumber \qquad   -2 f(x, z, a) + 2 f(x, w, b) -2f(y, z, b) +2 f(y, w, a) \| \\ &&  \quad  \le \left \|s \left(      f(x+y, z-w, a+b) + f(x-y, z+w, a-b) \right.\right.\\ &&  \nonumber\qquad \left.\left. -2 f(x, z, a) + 2 f(x, w, b) -2f(y, z, b) +2 f(y, w, a) \right)\right\|  \nonumber \\ && \qquad  +  \theta (\|x\|^r + \|y\|^r )( \|z\|^r + \|w\|^r) ( \|a\|^r + \|b\|^r)\nonumber 
\end{eqnarray}
for all $x, y, z, w,a,b \in X$.
Then there exists a unique  tri-additive mapping $L : X^3 \rightarrow Y$
such that
\begin{eqnarray}\label{3.5}
\|f(x, z,a)- L(x, z,a) \|  \le \frac{2^r \theta }{2(2^{ r}-2)} \|x \|^r  \|z \|^r \|a\|^r
\end{eqnarray} for all $x, z,a \in X$.
\end{theorem}

\begin{proof}
Letting $y=w=b=0$   in {\rm (\ref{3.4})}, we get
\begin{eqnarray}\label{3.6}
  \left\| 4 f\left(\frac{x}{2}, z, a\right) -2 f(x, z, a)\right\|  \le   \theta \|x\|^r \|z\|^r \|a\|^r
\end{eqnarray}
for all $x, z, a \in X$.  So
\begin{eqnarray}\label{3.7}
 \left\|2^{l} f\left(\frac{x}{2^l}, z, a\right) - 2^{m} f\left(\frac{x}{2^{m}}, z, a\right)\right\|  &  \le &
 \sum_{j=l}^{m-1} \left\| 2^j f\left(\frac{x}{2^j}, z, a\right) - 2^{j+1} f\left(\frac{x}{2^{j+1}}, z, a\right) \right\|
 \\ & \le & \sum_{j=l}^{m-1} \frac{2^{ j}\theta }{2^{ r j +1}}\|x\|^r \|z\|^r \|a\|^r \nonumber 
\end{eqnarray}
 for all nonnegative
integers $m$ and $l$ with $m>l$ and all $x, z, a \in X$. It follows from
{\rm (\ref{3.7})} that the sequence
$\{2^{k} f\left(\frac{x}{2^k}, z, a\right)\}$ is  Cauchy  for all $x, z, a \in X$.
Since $Y$ is a  Banach space, the sequence $\{2^{k} f\left(\frac{x}{2^k}, z, a\right)\}$ converges.
So one can define
the mapping $L : X^3 \rightarrow Y$ by
$$L(x, z,a) : = \lim_{k\to \infty} 2^{k} f\left(\frac{x}{2^k}, z,a\right) $$
for all $x, z,a \in X$.   Moreover, letting $l =0$ and passing the limit $m \to
\infty$ in {\rm (\ref{3.7})}, we get {\rm (\ref{3.5})}.

The rest of the proof is similar to the proof of Theorem \ref{thm2.3}.
\end{proof}

\begin{theorem}
Let $r < 1$ and $\theta$ be nonnegative real numbers    and let $f :
X^3 \rightarrow Y$ be a  mapping satisfying {\rm (\ref{3.4})} and $f(x,0,a) = f(0,z,a) =f(x,y,0)=0$ for all $x,z,a\in X$. 
Then there exists a unique  tri-additive mapping $L : X^3 \rightarrow Y$
such that
\begin{eqnarray}\label{3.8}
\|f(x, z,a) - L(x, z,a) \|  \le \frac{2^r\theta }{2(2-2^{ r})} \|x \|^r \|z \|^r \|a\|^r
\end{eqnarray} for all $x, z,a \in X$.
\end{theorem}

\begin{proof}
It follows from {\rm (\ref{3.6})} that
\begin{eqnarray*}
  \left\| f(x, z, a) - \frac{1}{2} f(2x, z, a) \right\|  \le  \frac{2^r\theta}{4} \|x\|^r \|z\|^r \|a\|^r
\end{eqnarray*}
for all $x, z, a \in X$.
Hence
 \begin{eqnarray}\label{3.9}
 \left\|\frac{1}{2^l} f(2^l x, z,a) - \frac{1}{2^m} f(2^m x, z,a)\right\|   & \le &
 \sum_{j=l}^{m-1} \left\| \frac{1}{2^j} f\left(2^j x, z,a\right) -\frac{1}{ 2^{j+1}} f\left(2^{j+1} x, z,a\right) \right\| \\ & \le & \nonumber   \sum_{j=l}^{m-1}  \frac{2^{r j} }{2^{j+2}}2^r \theta \|x\|^r \|z\|^r 
\end{eqnarray}
 for all nonnegative
integers $m$ and $l$ with $m>l$ and all $x, z,a \in X$.
 It follows from  {\rm (\ref{3.9})} that the sequence $\{\frac{1}{2^n}
f(2^n x, z,a)\}$ is a Cauchy sequence for all $x, z,a \in X$. Since
$Y$ is complete, the sequence $\{\frac{1}{2^n} f(2^n x, z,a)\}$
converges. So one can define the mapping $L : X^3 \rightarrow Y$ by
$$L(x, z,a) : = \lim_{n\to \infty} \frac{1}{2^n} f(2^n x, z,a) $$
for all $x, z,a \in X$. Moreover, letting $l =0$ and passing the limit $m \to
\infty$ in {\rm (\ref{3.9})}, we get {\rm (\ref{3.8})}.

The rest of the proof is similar to the proof of Theorem \ref{thm2.3}.
\end{proof}

\section{Permuting triderivations on Banach algebras}\label{sec:4}

Now, we investigate permuting triderivations on complex Banach algebras and unital $C^*$-algebras
 associated with the tri-additive $s$-functional  inequalities (\ref{0.1}) and (\ref{0.2}).

\begin{lemma} {\rm \cite[Lemma 2.1]{bp}} 
Let  $f: X^2 \to Y$ be a bi-additive mapping such that $f(\lambda x, \mu z) = \lambda \mu f(x, z)$ for all $x,z \in X$ and $\lambda,\mu \in {\mathbb{T}^1 } : = \{\nu \in {\mathbb{C}} ~:~ |\nu|=1\}$. Then $f$ is $\mathbb{C}$-bilinear.
\end{lemma}

\begin{lemma}  \label{lm2.1}
Let  $f: X^3 \to Y$ be a tri-additive mapping such that $f(\lambda x, \mu z, \eta a) = \lambda \mu \eta  f(x, z, a)$ for all $x,z, a \in X$ and $\lambda, \mu, \eta \in {\mathbb{T}^1 }$. Then $f$ is $\mathbb{C}$-trilinear.
\end{lemma}

\begin{proof}
The proof is similar to the proof of \cite[Lemma 2.1]{bp}.
\end{proof}

\begin{theorem}\label{thm4.1} 
Let $r > 2$ and $\theta$ be nonnegative real numbers, and let $f :
A^3 \rightarrow A$ be a  mapping satisfying   $f(x,0, a)=f(0,z, a) =f(x,z,0)=0$ and 
\begin{eqnarray}\label{4.1}
&& \nonumber \| f(\lambda(x+y), \mu (z-w), \eta(a+b)) + f(\lambda(x-y), \mu(z+w), \eta(a-b)) \\ &&  \nonumber\qquad -\lambda \mu \eta (2 f(x, z, a) - 2 f(x, w, b) +2f(y, z, b) - 2 f(y, w, a))\| \\ &&  \quad  \le \left \|s \left(2f\left(\frac{x+y}{2}, z-w, a+b \right) + 2f\left(\frac{x-y}{2}, z+w, a-b\right) \right. \right. \\ && \qquad \left. \left. -2 f(x, z, a) + 2 f(x, w, b) -2f(y, z, b) +2 f(y, w, a)\right)\right\| , \nonumber   \\ & & \qquad \qquad  +  
 \theta ( \|x\|^r + \|y\|^r) (\|z\|^r +\|w\|^r  )  (\|a\|^r +\|b\|^r  )\nonumber 
\end{eqnarray} for all $\lambda, \mu, \eta \in {\mathbb{T}}^1$  and   all $x,y,z, w, a, b \in A$.
Then there exists a unique $\mathbb{C}$-trilinear  mapping  $D : A^3 \rightarrow A$
such that  
\begin{eqnarray}\label{4.4}
\|f(x, z, a)- D(x, z, a) \|  \le \frac{  2\theta }{2^r-2} \|x \|^r \|z \|^r \|a\|^r
\end{eqnarray} for all $x, z, a \in A$.

If, in addition, the mapping $f: A^3 \to A$ satisfies 
\begin{eqnarray}\label{4.2}
\| f(x y, z, a) -  f(x, z, a)y - x f(y,z, a)  \|   \le  \theta ( \|x\|^r +  \|y\|^r )  \|z\|^r \|a\|^r ,
\end{eqnarray}
\begin{eqnarray}\label{5.02}
\|f(x_{\sigma(1)}, x_{\sigma(2)}, x_{\sigma(3)}) - f(x_1, x_2, x_3)\| \le \theta  \cdot \|x_1\|^r  \|x_2 \|^r \|x_3 \|^r 
\end{eqnarray}
 for  all $x, y, z, x_1, x_2, x_3, a \in A$ and for every permutation $(\sigma(1), \sigma(2), \sigma(3))$ of $(1,2,3)$, 
then the mapping $D: A^3 \to A$ is a permuting  triderivation. 
\end{theorem}

\begin{proof}
Let $\lambda =\mu =\eta =1$ in {\rm (\ref{4.1})}.
By Theorem \ref{thm2.3}, there is a unique tri-additive mapping $D: A^3 \to A$ satisfying (\ref{4.4}) defined by 
$$D(x, z, a) : = \lim_{n\to \infty} 2^n f\left(\frac{x}{2^n}, z, a\right) $$
for all $x, z, a \in A$.  

Letting $y=x$ and $w=b=0$ in (\ref{4.1}), we get 
$$ \|f(2\lambda x, \mu z, \eta a) - 2 \lambda \mu \eta  f(x,z, a)\| \le 2\theta \|x\|^r \|z\|^r \|a\|^r$$  for all $x,z, a\in A$ and all $\lambda, \mu, \eta \in {\mathbb{T}^1}$. 
So \begin{eqnarray*}
&& \|D(2\lambda x, \mu z, \eta a) - 2 \lambda \mu \eta  D(x,z, a)\|\\ && =
\lim_{n\to \infty} 2^n \left\|f\left(2\lambda \frac{x}{2^n}, \mu z, \eta a \right) - 2 \lambda \mu \eta  f\left(\frac{x}{2^n},z, a\right)\right \| \le \lim_{n\to\infty} \frac{2^{n+1}}{2^{rn}}\theta \|x\|^r \|z\|^r \|a\|^r =0
\end{eqnarray*}  for all $x,z, a\in A$ and all $\lambda, \mu, \eta \in {\mathbb{T}^1}$. 
Since $D$ is tri-additive, $D(\lambda x, \mu z, \eta a) = \lambda \mu \eta  D(x,z, a)$
 for all $x,z, a\in A$ and all $\lambda, \mu, \eta \in {\mathbb{T}^1}$. 
By Lemma \ref{lm2.1}, 
 the tri-additive mapping $D: A^3 \to A$ is $\mathbb{C}$-trilinear.

It follows from {\rm (\ref{4.2})} that
\begin{eqnarray*}
 &&  \| D (xy, z, a) - D(x, z, a)y - xD(y,z, a) \| \\ &&  =   \lim_{n\to\infty} 4^n \left\|
 f\left(\frac{xy}{2^n\cdot 2^n}, z, a\right) -  f\left(\frac{x}{2^n}, z, a\right)  \frac{y}{2^n}- \frac{x}{2^n} f\left(\frac{y}{2^n}, z, a\right)
 \right\|  \\ &&  \le    \lim_{n\to \infty} \frac{4^n \theta}{2^{rn}} ( \|x\|^r + \|y\|^r ) \|z\|^r \|a\|^r = 0
\end{eqnarray*}
for all $x, y, z, a \in A$. Thus $$D (xy, z, a) = D(x, z, a) y + xD(y, z, a) $$
for all $x, y, z, a\in A$.

It follows from {\rm (\ref{5.02})} that
\begin{eqnarray*}
 &&  \| D(x_{\sigma(1)}, x_{\sigma(2)}, x_{\sigma(3)}) - D(x_1, x_2, x_3)\| \\ &&  =   \lim_{n\to\infty} 8^n \left\|
f\left(\frac{x_{\sigma(1)}}{2^n}, \frac{x_{\sigma(2)}}{2^n}, \frac{x_{\sigma(3)}}{2^n}\right) - f\left(\frac{x_1}{2^n}, \frac{x_2}{2^n}, \frac{x_3}{2^n}\right)
 \right\|  \\ &&  \le    \lim_{n\to \infty} \frac{8^n \theta}{8^{rn}} \theta  \cdot \|x_1\|^r  \|x_2 \|^r \|x_3 \|^r  = 0
\end{eqnarray*}
for all $x_1, x_2, x_3 \in A$. Thus $$D(x_{\sigma(1)}, x_{\sigma(2)}, x_{\sigma(3)}) = D(x_1, x_2, x_3)$$
for all $x_1, x_2, x_3\in A$ and for every permutation $(\sigma(1), \sigma(2), \sigma(3))$ of $(1,2,3)$.
Hence the  mapping $D: A^3 \rightarrow
A$ is a permuting triderivation.
\end{proof}

In Theorem \ref{thm4.1}, if $f(2x, z, a) = 2 f(x,z,a)$ for all $x,z,a\in A$, then one can easily show that $D(x,z,a) = f(x,z,a)$ for all $x,z,a \in A$.
 Thus we can obtain the following corollary.

\begin{corollary}
Let $r > 2$ and $\theta$ be nonnegative real numbers, and let $f :
A^3 \rightarrow A$ be a  mapping satisfying   $f(x,0, a)=f(0,z, a) =f(x,z,0)=0$ and {\rm (\ref{4.1})}.
Then there exists a unique $\mathbb{C}$-trilinear  mapping  $D : A^3 \rightarrow A$
satisfying {\rm (\ref{4.4})}. 

If, in addition, the mapping $f: A^3 \to A$ satisfies {\rm (\ref{4.2})}, {\rm (\ref{5.02})} and $f(2x, z, a) = 2 f(x,z,a)$ for all $x,z,a\in A$, 
then the mapping $f: A^3 \to A$ is a permuting  triderivation. 
\end{corollary}

\begin{theorem}\label{thm4.2}
Let $r < 1$ and $\theta$ be nonnegative real numbers, and let $f :
A^3 \rightarrow A$ be a  mapping satisfying {\rm (\ref{4.1})} and  $f(x,0, a)=f(0,z, a) = f(x,z,0) =0$ for all $x,z, a\in A$. 
Then there exists a unique  $\mathbb{C}$-trilinear  mapping  $D : A^3 \rightarrow A$
such that
\begin{eqnarray}\label{4.5}
\|f(x, z, a)- D(x, z, a) \|  \le \frac{ 2\theta }{2- 2^r} \|x \|^r \|z \|^r \|a\|^r
\end{eqnarray} for all $x, z, a \in A$.

If, in addition, the mapping $f: A^3 \to A$ satisfies {\rm (\ref{4.2})} and  {\rm (\ref{5.02})}, 
then the mapping $D: A^3 \to A$ is a permuting  triderivation.  
\end{theorem}

\begin{proof}
The proof is similar to the proof of Theorem \ref{thm4.1}.
\end{proof}

\begin{corollary}
Let $r < 1$ and $\theta$ be nonnegative real numbers, and let $f :
A^3 \rightarrow A$ be a  mapping satisfying {\rm (\ref{4.1})} and  $f(x,0, a)=f(0,z, a) = f(x,z,0) =0$ for all $x,z, a\in A$. 
Then there exists a unique  $\mathbb{C}$-trilinear  mapping  $D : A^3 \rightarrow A$
satisfying {\rm (\ref{4.5})}. 

If, in addition, the mapping $f: A^3 \to A$ satisfies {\rm (\ref{4.2})}, {\rm (\ref{5.02})}  and $f(2x, z, a) = 2 f(x,z,a)$ for all $x,z,a\in A$, 
then the mapping $f: A^3 \to A$ is a permuting  triderivation.  
\end{corollary}

Similarly, we can obtain the following results.

\begin{theorem}\label{thm5.1}
Let $r > 2$ and $\theta$ be nonnegative real numbers, and let $f :
A^3 \rightarrow A$ be a  mapping satisfying $f(x,0, a)=f(0,z, a) = f(x,z,0)=0$ and 
\begin{eqnarray}\label{5.1}
 && \nonumber \left\|2f\left(\lambda\frac{x+y}{2}, \mu(z-w), \eta(a+b) \right) + 2f\left(\lambda\frac{x-y}{2}, \mu(z+w), \eta(a-b)\right) \right.  \\ && \nonumber \qquad   - \lambda \mu \eta (2 f(x, z, a) - 2 f(x, w, b) +2f(y, z, b) -2 f(y, w, a)) \| \\ &&  \quad  \le \left \|s \left(      f(x+y, z-w, a+b) + f(x-y, z+w, a-b) \right.\right.\\ &&  \nonumber\qquad \left.\left. -2 f(x, z, a) + 2 f(x, w, b) -2f(y, z, b) +2 f(y, w, a) \right)\right\|  \nonumber \\ && \qquad  +  \theta (\|x\|^r + \|y\|^r )( \|z\|^r + \|w\|^r) ( \|a\|^r + \|b\|^r)\nonumber 
\end{eqnarray} for all $\lambda, \mu, \eta \in {\mathbb{T}}^1$  and   all $x,y,z, w, a, b \in A$.
Then there exists a unique  $\mathbb{C}$-trilinear mapping  $D : A^3 \rightarrow A$
such that
\begin{eqnarray}\label{5.4}
\|f(x, z, a)- D(x, z, a) \|  \le \frac{2^r \theta }{2(2^{ r}-2)} \|x \|^r  \|z \|^r \|a\|^r
\end{eqnarray} for all $x, z, a \in A$.

If, in addition, the mapping $f: A^3 \to A$ satisfies {\rm (\ref{4.2})} and  {\rm (\ref{5.02})},
then the mapping $D: A^3 \to A$ is a permuting triderivation.  
\end{theorem}

\begin{corollary}
Let $r > 2$ and $\theta$ be nonnegative real numbers, and let $f :
A^3 \rightarrow A$ be a  mapping satisfying $f(x,0, a)=f(0,z, a) = f(x,z,0)=0$ and {\rm (\ref{5.1})}. 
Then there exists a unique  $\mathbb{C}$-trilinear mapping  $D : A^3 \rightarrow A$
satisfying {\rm (\ref{5.4})}. 

If, in addition, the mapping $f: A^3 \to A$ satisfies {\rm (\ref{4.2})},  {\rm (\ref{5.02})}  and $f(2x, z, a) = 2 f(x,z,a)$ for all $x,z,a\in A$, 
then the mapping $f: A^3 \to A$ is a permuting triderivation.  
\end{corollary}

\begin{theorem}\label{thm5.2}
Let $r < 1$ and $\theta$ be nonnegative real numbers, and let $f :
A^3 \rightarrow A$ be a  mapping satisfying {\rm (\ref{5.1})} and  $f(x,0, a)=f(0,z, a) = f(x,z,0)=0$ for all $x,z, a\in A$.
Then there exists a unique $\mathbb{C}$-trilinear  mapping  $D : A^3 \rightarrow A$
such that
\begin{eqnarray}\label{5.5}
\|f(x, z, a) - D(x, z, a) \|  \le \frac{2^r\theta }{2(2-2^{ r})} \|x \|^r \|z \|^r\|a\|^r
\end{eqnarray} for all $x, z, a \in A$.

If, in addition, the mapping $f: A^3 \to A$ satisfies {\rm (\ref{4.2})} and  {\rm (\ref{5.02})},
then the mapping $D: A^3 \to A$ is a permuting triderivation. 
\end{theorem}

\begin{corollary}
Let $r < 1$ and $\theta$ be nonnegative real numbers, and let $f :
A^3 \rightarrow A$ be a  mapping satisfying {\rm (\ref{5.1})} and  $f(x,0, a)=f(0,z, a) = f(x,z,0)=0$ for all $x,z, a\in A$.
Then there exists a unique $\mathbb{C}$-trilinear  mapping  $D : A^3 \rightarrow A$
satisfying {\rm (\ref{5.5})}.

If, in addition, the mapping $f: A^3 \to A$ satisfies {\rm (\ref{4.2})}, {\rm (\ref{5.02})} and $f(2x, z, a) = 2 f(x,z,a)$ for all $x,z,a\in A$, 
then the mapping $f: A^3 \to A$ is a permuting triderivation. 
\end{corollary}

From now on, assume that  $A$  is  a  unital $C^*$-algebra with unit $e$ and unitary group $U(A)$. 

\begin{theorem}\label{thm4.3}
Let $r > 2$ and $\theta$ be nonnegative real numbers, and let $f :
A^3 \rightarrow A$ be a mapping satisfying {\rm (\ref{4.1})} and  $f(x,0, a)=f(0,z, a) = f(x,z,0) =0$ for all $x,z, a\in A$.
Then there exists a unique $\mathbb{C}$-trilinear mapping  $D : A^3 \rightarrow A$ 
satisfying {\rm (\ref{4.4})}. 

If, in addition, the mapping $f: A^3\to A$ satisfies {\rm (\ref{5.02})} and 
\begin{eqnarray}\label{5.2}
\| f(uy, z, a) - f(u, z, a) y - uf(y,z, a) \|   \le  \theta (1+ \|y\|^r)\|z\|^r \|a\|^r
\end{eqnarray}
 for  all $u \in U(A)$ and all $y,z,a\in A$,
then the mapping $D: A^3 \to A$  is a permuting triderivation.
\end{theorem}

\begin{proof}
By the same reasoning as in the proof of Theorem \ref{thm4.1},
there is a unique $\mathbb{C}$-trilinear  mapping $D: A^3 \to A$ satisfying (\ref{4.4}) defined by 
$$D(x, z, a) : = \lim_{n\to \infty} 2^n f\left(\frac{x}{2^n}, z, a\right) $$
for all $x, z, a \in A$.

By the same reasoning as in the proof of Theorem \ref{thm4.1}, $ D(uy, z, a) = D(u,z, a)y + uD(y,z, a)$ for all $u\in U(A)$ and all $y,z, a\in A$. 
 Since $D$ is  $\mathbb
{C}$-linear in the first variable and each  $x \in A$ is a  finite linear combination of
unitary elements (see \cite{kr83}),  i.e., $x = \sum_{j=1}^m \lambda_j u_j \
(\lambda_j \in {\mathbb{C}}, $ $u_j \in U(A))$, 
 \begin{eqnarray*}
D(xy, z,a) & = & D (\sum_{j=1}^m \lambda_j  u_j y,  z,a ) = \sum_{j=1}^m  \lambda_j D(u_j y, z,a) = \sum_{j=1}^m 
\lambda_j  (D(u_j,z,a) y + u_j D(y, z,a) )\\ &= & (\sum_{j=1}^m  \lambda_j) D(u_j ,z,a) y + (\sum_{j=1}^m  \lambda_j u_j) D(y,z,a) 
= D(x,z,a) y + x D(y,z,a)
\end{eqnarray*} for all $x, y,z,a \in A$. 
 Thus the mapping $D : A^3 \to A$ is a permuting  triderivation.
\end{proof}

\begin{remark}
By a similar method in the proof of Theorem \ref{thm4.3}, one can obtain that if {\rm (\ref{5.2})} in Theorem \ref{thm4.3} is replaced by 
 \begin{eqnarray*}
\| f(uv, z, a) - f(u, z, a) v - uf(y,z, a) \|   \le  2 \theta\|z\|^r \|a\|^r
\end{eqnarray*}
 for  all $u, v \in U(A)$ and all $z,a\in A$, then  the mapping $D : A^3 \to A$ is a permuting  triderivation.
\end{remark}

\begin{corollary}
Let $r > 2$ and $\theta$ be nonnegative real numbers, and let $f :
A^3 \rightarrow A$ be a mapping satisfying {\rm (\ref{4.1})} and  $f(x,0, a)=f(0,z, a) = f(x,z,0) =0$ for all $x,z, a\in A$.
Then there exists a unique $\mathbb{C}$-trilinear mapping  $D : A^3 \rightarrow A$ 
satisfying {\rm (\ref{4.4})}. 

If, in addition, the mapping $f: A^3\to A$ satisfies {\rm (\ref{5.02})},  {\rm (\ref{5.2})} and $f(2x, z, a) = 2 f(x,z,a)$ for all $x,z,a\in A$, 
then the mapping $f: A^3 \to A$  is a permuting triderivation.
\end{corollary}

\begin{theorem}\label{thm4.4}
Let $r < 1$ and $\theta$ be nonnegative real numbers, and let $f :
A^3 \rightarrow A$ be a  mapping satisfying {\rm (\ref{4.1})}  and  $f(x,0,a)=f(0,z,a) =f(x,z,0)=0$ for all $x,z,a\in A$. 
Then there exists a unique $\mathbb{C}$-trilinear mapping  $D : A^3 \rightarrow A$
satisfying {\rm (\ref{4.5})}.

If, in addition, the mapping $f: A^3 \to A$ satisfies {\rm (\ref{5.02})} and {\rm (\ref{5.2})}, 
then the mapping $D: A^3 \to A$  is a permuting  triderivation.
\end{theorem}

\begin{proof}
The proof is similar to the proof of Theorem \ref{thm4.3}.
\end{proof}

\begin{corollary}
Let $r < 1$ and $\theta$ be nonnegative real numbers, and let $f :
A^3 \rightarrow A$ be a  mapping satisfying {\rm (\ref{4.1})}  and  $f(x,0,a)=f(0,z,a) =f(x,z,0)=0$ for all $x,z,a\in A$. 
Then there exists a unique $\mathbb{C}$-trilinear mapping  $D : A^3 \rightarrow A$
satisfying {\rm (\ref{4.5})}.

If, in addition, the mapping $f: A^3 \to A$ satisfies {\rm (\ref{5.02})},  {\rm (\ref{5.2})} and $f(2x, z, a) = 2 f(x,z,a)$ for all $x,z,a\in A$, 
then the mapping $f: A^3 \to A$  is a permuting  triderivation.
\end{corollary}

Similarly, we can obtain the following results.

\begin{theorem}\label{thm5.3}
Let $r > 2$ and $\theta$ be nonnegative real numbers, and let $f :
A^3 \rightarrow A$ be a  mapping satisfying {\rm (\ref{5.1})}  and  $f(x,0,a)=f(0,z,a) =f(x,z,0)=0$ for all $x,z,a\in A$. 
Then there exists a unique $\mathbb{C}$-trilinear mapping  $D : A^3 \rightarrow A$
satisfying {\rm (\ref{5.4})}. 

If, in addition, the mapping $f: A^3 \to A$ satisfies {\rm (\ref{5.02})} and  {\rm (\ref{5.2})},
then the mapping $D: A^3 \to A$  is a  permuting triderivation.
\end{theorem}

\begin{corollary}
Let $r > 2$ and $\theta$ be nonnegative real numbers, and let $f :
A^3 \rightarrow A$ be a  mapping satisfying {\rm (\ref{5.1})}  and  $f(x,0,a)=f(0,z,a) =f(x,z,0)=0$ for all $x,z,a\in A$. 
Then there exists a unique $\mathbb{C}$-trilinear mapping  $D : A^3 \rightarrow A$
satisfying {\rm (\ref{5.4})}. 

If, in addition, the mapping $f: A^3 \to A$ satisfies {\rm (\ref{5.02})},  {\rm (\ref{5.2})} and $f(2x, z, a) = 2 f(x,z,a)$ for all $x,z,a\in A$, 
then the mapping $f: A^3 \to A$  is a  permuting triderivation.
\end{corollary}

\begin{theorem}\label{thm5.4}
Let $r < 1$ and $\theta$ be nonnegative real numbers, and let $f :
A^3 \rightarrow A$ be a mapping satisfying {\rm (\ref{5.1})}   and  $f(x,0,a)=f(0,z,a)  = f(x,z,0)=0$ for all $x,z,a\in A$. 
Then there exists a unique $\mathbb{C}$-trilinear  mapping  $D : A^3 \rightarrow A$ 
satisfying {\rm (\ref{5.5})}. 

If, in addition, the mapping $f: A^3\to A$ satisfies {\rm (\ref{5.02})} and  {\rm (\ref{5.2})},
then  the mapping $D: A^3 \to A$  is a permuting  triderivation.
\end{theorem}

\begin{corollary}
Let $r < 1$ and $\theta$ be nonnegative real numbers, and let $f :
A^3 \rightarrow A$ be a mapping satisfying {\rm (\ref{5.1})}   and  $f(x,0,a)=f(0,z,a)  = f(x,z,0)=0$ for all $x,z,a\in A$. 
Then there exists a unique $\mathbb{C}$-trilinear  mapping  $D : A^3 \rightarrow A$ 
satisfying {\rm (\ref{5.5})}. 

If, in addition, the mapping $f: A^3\to A$ satisfies {\rm (\ref{5.02})},   {\rm (\ref{5.2})} and $f(2x, z, a) = 2 f(x,z,a)$ for all $x,z,a\in A$, 
then  the mapping $f: A^3 \to A$  is a permuting  triderivation.
\end{corollary}

\section{Permuting trihomomorphisms in Banach algebras}\label{sec:5}

Now, we investigate permuting trihomomorphisms in complex Banach algebras and unital $C^*$-algebras
 associated with the tri-additive $s$-functional  inequalities (\ref{0.1}) and (\ref{0.2}).

\begin{theorem}\label{thm8.1}
Let $r > 2$ and $\theta$ be nonnegative real numbers, and let $f :
A^3 \rightarrow B$ be a  mapping satisfying $f(x,0, a)=f(0,z, a) = f(x,z,0)=0$ and {\rm (\ref{4.1})}. 
Then there exists a unique $\mathbb{C}$-trilinear   mapping  $H : A^3 \rightarrow B$
satisfying {\rm (\ref{4.4})}, where $D$ is replaced by $H$ in  {\rm (\ref{4.4})}.

If, in addition, the mapping $f: A^3 \to B$ satisfies {\rm (\ref{5.02})} and 
\begin{eqnarray}\label{8.2}
\| f(x y, zw, ab) -  f(x, z,a) f(y,w, b)  \|   \le  \theta ( \|x\|^r +  \|y\|^r ) ( \|z\|^r +  \|w\|^r )  ( \|a\|^r +  \|b\|^r )  
\end{eqnarray}
 for  all $x, y, z,w,a,b \in A$,
then the mapping $H: A^3 \to B$ is a permuting   trihomomorphism. 
\end{theorem}

\begin{proof}
By the same reasoning as in the proof of Theorem \ref{thm4.1}, there is a unique $\mathbb{C}$-trilinear  mapping $H: A^3 \to B$, which is defined by 
$$H(x,z,a) = \lim_{n\to \infty} 2^n f\left(\frac{x}{2^n}, z,a\right)$$ for all $x,z,a\in A$.

It follows from {\rm (\ref{8.2})} that
\begin{eqnarray*}
 &&  \| H (xy, zw, ab) - H(x, z, a)H(y,w, b) \| \\ &&  =   \lim_{n\to\infty} 4^n \left\|
 f\left(\frac{xy}{2^n\cdot 2^n}, zw, ab\right) -  f\left(\frac{x}{2^n}, z, a\right) f\left(\frac{y}{2^n}, w, b\right)
 \right\|  \\ &&  \le    \lim_{n\to \infty} \frac{4^n \theta}{2^{rn}} ( \|x\|^r +  \|y\|^r ) ( \|z\|^r +  \|w\|^r )  ( \|a\|^r +  \|b\|^r )  = 0
\end{eqnarray*}
for all $x, y, z,w,a, b \in A$. Thus $$H (xy, zw, ab) = H(x, z,a) H(y, w,b) $$
for all $x, y, z, w, a, b \in A$.

The rest of the proof is similar to the proof of Theorem \ref{thm4.1}.
Hence the  mapping $H: A^3 \rightarrow
B$ is a permuting trihomomorphism.
\end{proof}

In Theorem \ref{thm8.1}, if $f(2x, z, a) = 2 f(x,z,a)$ for all $x,z,a\in A$, then one can easily show that $H(x,z,a) = f(x,z,a)$ for all $x,z,a \in A$.
 Thus we can obtain the following corollary.

\begin{corollary}
Let $r > 2$ and $\theta$ be nonnegative real numbers, and let $f :
A^3 \rightarrow B$ be a  mapping satisfying $f(x,0, a)=f(0,z, a) = f(x,z,0)=0$ and {\rm (\ref{4.1})}. 
Then there exists a unique $\mathbb{C}$-trilinear   mapping  $H : A^3 \rightarrow B$
satisfying {\rm (\ref{4.4})}, where $D$ is replaced by $H$ in  {\rm (\ref{4.4})}.

If, in addition, the mapping $f: A^3 \to B$ satisfies {\rm (\ref{5.02})}, {\rm (\ref{8.2})}   and $f(2x, z, a) = 2 f(x,z,a)$ for all $x,z,a\in A$, 
then the mapping $f: A^3 \to B$ is a permuting   trihomomorphism. 
\end{corollary}

\begin{theorem}\label{thm8.2}
Let $r < 1$ and $\theta$ be nonnegative real numbers, and let $f :
A^3 \rightarrow B$ be a  mapping satisfying {\rm (\ref{4.1})} and  $f(x,0,a)=f(0,z,a) =f(x,z,0)=0$ for all $x,z,a\in A$. 
Then there exists a unique  $\mathbb{C}$-trilinear  mapping  $H : A^3 \rightarrow B$
satisfying  {\rm (\ref{4.5})}, where $D$ is replaced by $H$ in  {\rm (\ref{4.5})}.

If, in addition, the mapping $f: A^3 \to B$ satisfies  {\rm (\ref{5.02})} and  {\rm (\ref{8.2})},
then the mapping $H: A^3 \to B$ is a permuting trihomomorphism.  
\end{theorem}

\begin{proof}
The proof is similar to the proof of Theorem \ref{thm8.1}.
\end{proof}

\begin{corollary}
Let $r < 1$ and $\theta$ be nonnegative real numbers, and let $f :
A^3 \rightarrow B$ be a  mapping satisfying {\rm (\ref{4.1})} and  $f(x,0,a)=f(0,z,a) =f(x,z,0)=0$ for all $x,z,a\in A$. 
Then there exists a unique  $\mathbb{C}$-trilinear  mapping  $H : A^3 \rightarrow B$
satisfying  {\rm (\ref{4.5})}, where $D$ is replaced by $H$ in  {\rm (\ref{4.5})}.

If, in addition, the mapping $f: A^3 \to B$ satisfies  {\rm (\ref{5.02})},  {\rm (\ref{8.2})}  and $f(2x, z, a) = 2 f(x,z,a)$ for all $x,z,a\in A$, 
then the mapping $f: A^3 \to B$ is a permuting trihomomorphism.  
\end{corollary}

Similarly, we can obtain the following results.

\begin{theorem}\label{thm10.1}
Let $r > 2$ and $\theta$ be nonnegative real numbers, and let $f :
A^3 \rightarrow B$ be a  mapping satisfying $f(x,0,a)=f(0,z,a)= f(x,z,0)=0$ and {\rm (\ref{5.1})}. 
Then there exists a unique  $\mathbb{C}$-trilinear mapping  $H : A^3 \rightarrow B$
satisfying  {\rm (\ref{5.4})}, where $D$ is replaced by $H$ in  {\rm (\ref{5.4})}.

If, in addition, the mapping $f: A^3 \to B$ satisfies  {\rm (\ref{5.02})} and  {\rm (\ref{8.2})},
then the mapping $H: A^3 \to B$ is a permuting trihomomorphism.  
\end{theorem}

\begin{corollary}
Let $r > 2$ and $\theta$ be nonnegative real numbers, and let $f :
A^3 \rightarrow B$ be a  mapping satisfying $f(x,0,a)=f(0,z,a)= f(x,z,0)=0$ and {\rm (\ref{5.1})}. 
Then there exists a unique  $\mathbb{C}$-trilinear mapping  $H : A^3 \rightarrow B$
satisfying  {\rm (\ref{5.4})}, where $D$ is replaced by $H$ in  {\rm (\ref{5.4})}.

If, in addition, the mapping $f: A^3 \to B$ satisfies  {\rm (\ref{5.02})},  {\rm (\ref{8.2})} and $f(2x, z, a) = 2 f(x,z,a)$ for all $x,z,a\in A$, 
then the mapping $f: A^3 \to B$ is a permuting trihomomorphism.  
\end{corollary}

\begin{theorem}\label{thm10.2}
Let $r < 1$ and $\theta$ be nonnegative real numbers, and let $f :
A^3 \rightarrow B$ be a  mapping satisfying {\rm (\ref{5.1})} and  $f(x,0,a)=f(0,z,a) =f(x,z,0)=0$ for all $x,z, a\in A$.
Then there exists a unique $\mathbb{C}$-trilinear  mapping  $H : A^3 \rightarrow B$
satisfying  {\rm (\ref{5.5})}, where $D$ is replaced by $H$ in  {\rm (\ref{5.5})}.

If, in addition, the mapping $f: A^3 \to B$ satisfies  {\rm (\ref{5.02})} and  {\rm (\ref{8.2})},
then the mapping $H: A^3 \to B$ is a permuting trihomomorphism. 
\end{theorem}

\begin{corollary}
Let $r < 1$ and $\theta$ be nonnegative real numbers, and let $f :
A^3 \rightarrow B$ be a  mapping satisfying {\rm (\ref{5.1})} and  $f(x,0,a)=f(0,z,a) =f(x,z,0)=0$ for all $x,z, a\in A$.
Then there exists a unique $\mathbb{C}$-trilinear  mapping  $H : A^3 \rightarrow B$
satisfying  {\rm (\ref{5.5})}, where $D$ is replaced by $H$ in  {\rm (\ref{5.5})}.

If, in addition, the mapping $f: A^3 \to B$ satisfies  {\rm (\ref{5.02})},  {\rm (\ref{8.2})} and $f(2x, z, a) = 2 f(x,z,a)$ for all $x,z,a\in A$, 
then the mapping $f: A^3 \to B$ is a permuting trihomomorphism. 
\end{corollary}

From now on, assume that  $A$  is  a  unital $C^*$-algebra with unit $e$ and unitary group $U(A)$. 

\begin{theorem}\label{thm10.3}
Let $r > 2$ and $\theta$ be nonnegative real numbers, and let $f :
A^3 \rightarrow B$ be a mapping satisfying {\rm (\ref{4.1})} and  $f(x,0,a)=f(0,z,a) = f(x,z,0)=0$ for all $x,z, a\in A$.
Then there exists a unique $\mathbb{C}$-trilinear mapping  $H : A^3 \rightarrow B$ 
satisfying {\rm (\ref{4.4})},  where $D$ is replaced by $H$ in  {\rm (\ref{4.4})}.

If, in addition, the mapping $f: A^3\to B$ satisfies  {\rm (\ref{5.02})}   and 
\begin{eqnarray}\label{10.2}
\| f(uy, zw, ab) - f(u, z,a) f(y,w,b) \|   \le  \theta (1+ \|y\|^r) ( \|z\|^r +  \|w\|^r )  ( \|a\|^r +  \|b\|^r )
\end{eqnarray}
 for  all $u \in U(A)$ and all $y,z,a\in A$,
then the mapping $H: A^3 \to B$  is a permuting trihomomorphism.
\end{theorem}

\begin{proof}
By the same reasoning as in the proof of Theorem \ref{thm4.1},
there is a unique $\mathbb{C}$-trilinear  mapping $H: A^3 \to B$ satisfying (\ref{4.4}) defined by 
$$H(x, z,a) : = \lim_{n\to \infty} 2^n f\left(\frac{x}{2^n}, z,a\right) $$
for all $x, z,a \in A$.

By the same reasoning as in the proof of Theorem \ref{thm8.1}, $ H(uy, zw,ab) = H(u,z,a)H(y,w,b)$ for all $u\in U(A)$ and all $y,z,w, a, b\in A$. 
 Since $H$ is  $\mathbb
{C}$-linear in the first variable and each  $x \in A$ is a  finite linear combination of
unitary elements (see \cite{kr83}),  i.e., $x = \sum_{j=1}^m \lambda_j u_j \
(\lambda_j \in {\mathbb{C}}, $ $u_j \in U(A))$, 
 \begin{eqnarray*}
H(xy, zw, ab) & = & H (\sum_{j=1}^m \lambda_j  u_j y,  zw,ab ) = 
\sum_{j=1}^m  \lambda_j H(u_j y, zw,ab) = \sum_{j=1}^m 
\lambda_j  (H(u_j,z,a) H(y, w,b) )\\ &= & (\sum_{j=1}^m  \lambda_j) H(u_j ,z,a)H(y,w,b) 
= H(x,z,a) H(y,w,b)
\end{eqnarray*} for all $x, y,z,w,a,b \in A$. 
 Thus the mapping $H : A^3 \to B$ is a  permuting trihomomorphism.
\end{proof}

\begin{remark}
By a similar method in the proof of Theorem \ref{thm10.3}, one can obtain that if {\rm (\ref{10.2})} in Theorem \ref{thm10.3} is replaced by 
\begin{eqnarray*}
\| f(u_1 u_2, u_3 u_4, u_5 u_6) - f(u_1,u_3, u_5) f(u_2, u_4, u_6) \|   \le  8 \theta 
\end{eqnarray*}
 for  all $u_1, u_2, u_3, u_4, u_5, u_6 \in U(A)$, then  the mapping $H : A^3 \to B$ is a permuting  trihomomorphism.
\end{remark}

\begin{corollary}
Let $r > 2$ and $\theta$ be nonnegative real numbers, and let $f :
A^3 \rightarrow B$ be a mapping satisfying {\rm (\ref{4.1})} and  $f(x,0,a)=f(0,z,a) = f(x,z,0)=0$ for all $x,z, a\in A$.
Then there exists a unique $\mathbb{C}$-trilinear mapping  $H : A^3 \rightarrow B$ 
satisfying {\rm (\ref{4.4})},  where $D$ is replaced by $H$ in  {\rm (\ref{4.4})}.

If, in addition, the mapping $f: A^3\to B$ satisfies  {\rm (\ref{5.02})}, {\rm (\ref{10.2})}   and $f(2x, z, a) = 2 f(x,z,a)$ for all $x,z,a\in A$, 
then the mapping $f: A^3 \to B$  is a permuting trihomomorphism.
\end{corollary}

\begin{theorem}\label{thm10.4}
Let $r < 1$ and $\theta$ be nonnegative real numbers, and let $f :
A^3 \rightarrow B$ be a  mapping satisfying {\rm (\ref{4.1})}  and  $f(x,0,a)=f(0,z,a) =f(x,z,0)=0$ for all $x,z,a\in A$. 
Then there exists a unique $\mathbb{C}$-trilinear mapping  $H : A^3 \rightarrow B$
satisfying {\rm (\ref{4.5})},  where $D$ is replaced by $H$ in  {\rm (\ref{4.5})}.

If, in addition, the mapping $f: A^3 \to B$ satisfies   {\rm (\ref{5.02})} and {\rm (\ref{10.2})},
then the mapping $H: A^3 \to B$  is a permuting  trihomomorphism.
\end{theorem}

\begin{proof}
The proof is similar to the proof of Theorem \ref{thm10.3}.
\end{proof}

\begin{corollary}
Let $r < 1$ and $\theta$ be nonnegative real numbers, and let $f :
A^3 \rightarrow B$ be a  mapping satisfying {\rm (\ref{4.1})}  and  $f(x,0,a)=f(0,z,a) =f(x,z,0)=0$ for all $x,z,a\in A$. 
Then there exists a unique $\mathbb{C}$-trilinear mapping  $H : A^3 \rightarrow B$
satisfying {\rm (\ref{4.5})},  where $D$ is replaced by $H$ in  {\rm (\ref{4.5})}.

If, in addition, the mapping $f: A^3 \to B$ satisfies   {\rm (\ref{5.02})},  {\rm (\ref{10.2})}  and $f(2x, z, a) = 2 f(x,z,a)$ for all $x,z,a\in A$, 
then the mapping $f: A^3 \to B$  is a permuting  trihomomorphism.
\end{corollary}

Similarly, we can obtain the following results.

\begin{theorem}\label{thm10.8}
Let $r > 2$ and $\theta$ be nonnegative real numbers, and let $f :
A^3 \rightarrow B$ be a  mapping satisfying {\rm (\ref{5.1})}  and  $f(x,0,a)=f(0,z,a) =f(x,z,0)=0$ for all $x,z,a\in A$. 
Then there exists a unique $\mathbb{C}$-trilinear mapping  $H : A^3 \rightarrow B$
satisfying {\rm (\ref{5.4})},  where $D$ is replaced by $H$ in  {\rm (\ref{5.4})}.

If, in addition, the mapping $f: A^3 \to B$ satisfies  {\rm (\ref{5.02})} and  {\rm (\ref{10.2})},
then the mapping $H: A^3 \to B$  is a permuting trihomomorphism.
\end{theorem}

\begin{corollary}
Let $r > 2$ and $\theta$ be nonnegative real numbers, and let $f :
A^3 \rightarrow B$ be a  mapping satisfying {\rm (\ref{5.1})}  and  $f(x,0,a)=f(0,z,a) =f(x,z,0)=0$ for all $x,z,a\in A$. 
Then there exists a unique $\mathbb{C}$-trilinear mapping  $H : A^3 \rightarrow B$
satisfying {\rm (\ref{5.4})},  where $D$ is replaced by $H$ in  {\rm (\ref{5.4})}.

If, in addition, the mapping $f: A^3 \to B$ satisfies  {\rm (\ref{5.02})},   {\rm (\ref{10.2})}  and $f(2x, z, a) = 2 f(x,z,a)$ for all $x,z,a\in A$, 
then the mapping $f: A^3 \to B$  is a permuting trihomomorphism.
\end{corollary}

\begin{theorem}\label{thm10.9}
Let $r < 1$ and $\theta$ be nonnegative real numbers, and let $f :
A^3 \rightarrow B$ be a mapping satisfying {\rm (\ref{5.1})}   and  $f(x,0,a)=f(0,z,a) = f(x,z,0) =0$ for all $x,z,a\in A$. 
Then there exists a unique $\mathbb{C}$-trilinear  mapping  $H : A^3 \rightarrow B$ 
satisfying {\rm (\ref{5.5})},  where $D$ is replaced by $H$ in  {\rm (\ref{5.5})}.

If, in addition, the mapping $f: A^3\to B$ satisfies  {\rm (\ref{5.02})} and  {\rm (\ref{10.2})},
then  the mapping $H: A^3 \to B$  is a permuting  trihomomorphism.
\end{theorem}

\begin{corollary}
Let $r < 1$ and $\theta$ be nonnegative real numbers, and let $f :
A^3 \rightarrow B$ be a mapping satisfying {\rm (\ref{5.1})}   and  $f(x,0,a)=f(0,z,a) = f(x,z,0) =0$ for all $x,z,a\in A$. 
Then there exists a unique $\mathbb{C}$-trilinear  mapping  $H : A^3 \rightarrow B$ 
satisfying {\rm (\ref{5.5})},  where $D$ is replaced by $H$ in  {\rm (\ref{5.5})}.

If, in addition, the mapping $f: A^3\to B$ satisfies  {\rm (\ref{5.02})},  {\rm (\ref{10.2})}  and $f(2x, z, a) = 2 f(x,z,a)$ for all $x,z,a\in A$, 
then  the mapping $f: A^3 \to B$  is a permuting  trihomomorphism.
\end{corollary}

\section*{Conclusions}

We have  introduced and solved  the tri-additive $s$-functional inequalities  {\rm (\ref{0.1})} and {\rm (\ref{0.2})} and  proved the Hyers-Ulam stability and the hyperstability of   permuting triderivations and permuting  trihomomorphisms in  Banach algebras and unital $C^*$-algebras, 
associated with  the tri-additive $s$-functional inequalities  {\rm (\ref{0.1})} and {\rm (\ref{0.2})}.

\section*{Acknowledgments}

This work  was supported by Basic Science Research Program through the National Research Foundation of Korea funded by the Ministry of Education, Science and Technology (NRF-2017R1D1A1B04032937).

\section*{Competing interests}

The author declares that he has  no competing interests.

\end{document}